\documentclass[11pt]{article}
\usepackage{graphicx}
\usepackage{amsmath,amssymb,amsthm}
\usepackage[all]{xy}

\newtheorem{theorem}{Theorem}[section]
\newtheorem{lemma}[theorem]{Lemma}
\newtheorem{proposition}[theorem]{Proposition}

\theoremstyle{definition}
\newtheorem{definition}[theorem]{Definition}
\newtheorem{remark}[theorem]{Remark}

\newtheorem*{acknowledgements}{Acknowledgements}

\numberwithin{equation}{section}

\renewcommand{\phi}{\varphi}

\newcommand{\Aff}{\operatorname{Aff}}
\newcommand{\Aut}{\operatorname{Aut}}
\newcommand{\Coker}{\operatorname{Coker}}
\newcommand{\Ext}{\operatorname{Ext}}
\newcommand{\Hom}{\operatorname{Hom}}

\renewcommand{\Im}{\operatorname{Im}}
\newcommand{\id}{\operatorname{id}}
\newcommand{\Ker}{\operatorname{Ker}}
\newcommand{\sgn}{\operatorname{sgn}}
\newcommand{\simp}{\operatorname{simp}}

\newcommand{\N}{\mathbb{N}}
\newcommand{\Z}{\mathbb{Z}}

\newcommand{\R}{\mathbb{R}}
\newcommand{\C}{\mathbb{C}}
\newcommand{\T}{\mathbb{T}}

\newcommand{\G}{\mathcal{G}}
\renewcommand{\H}{\mathcal{H}}

\title{Cup and cap products for cohomology and homology groups of 
ample groupoids}
\author{Hiroki Matui 
\thanks{This work was supported by JSPS KAKENHI Grant Number 23K22397.}\\
Graduate School of Science \\
Chiba University \\
Inage-ku, Chiba 263-8522, Japan\\
\and
Takehiko Mori \\
Graduate School of Science and Engineering\\
Chiba University \\
Inage-ku, Chiba 263-8522, Japan
}
\date{}

\begin{document}
\maketitle

\begin{abstract}
This paper explores the cup and cap products 
within the cohomology and homology groups of ample groupoids, 
focusing on their applications and fundamental properties. 
Ample groupoids, 
which are \'etale groupoids with a totally disconnected unit space, 
play a crucial role in the study of 
topological dynamical systems and operator algebras. 
We introduce the cup product, which defines a bilinear map 
on cohomology classes, providing a graded ring structure, 
and the cap product, which defines a bilinear map 
relating homology and cohomology. 
The paper aims to make these concepts accessible to 
a broader mathematical audience, 
offering clear definitions and detailed explanations. 
We also demonstrate an application of the cap product 
in the analysis of automorphisms of groupoid $C^*$-algebras. 
Specifically, we show 
how it helps determine the asymptotic innerness of automorphisms. 
Our results include the first explicit computations of cup products 
in the cohomology of tiling spaces, 
which may pave the way for new research in this area. 
\end{abstract}

\section{Introduction}

The study of ample groupoids and their cohomology and homology groups 
has emerged as a significant area of research and 
has numerous applications in the analysis of dynamical systems, 
operator algebras, and noncommutative geometry. 
Topological groupoids and their cohomology theory were introduced 
by Renault in \cite{Re_text}, 
providing a framework to understand the interplay 
between topological dynamical systems and $C^*$-algebras. 
A topological groupoid is said to be \'etale 
when the range and source maps are local homeomorphisms, 
and an \'etale groupoid is said to be ample 
when its unit space is totally disconnected. 
The homology theory of ample groupoids has been studied extensively, 
particularly following the foundational works of the first named author 
\cite{Ma12PLMS,Ma15crelle}. 
In particular, 
significant progress has been made regarding the HK conjecture, 
which examines the relationship between the homology groups and 
the $K$-groups of $C^*$-algebras. 
Numerous results have been obtained by various researchers in this direction, 
e.g.\ \cite{Li15JFA,Ma16Adv,FKPS19Munster,Or20JNG,Yi20BullAust,
OS22MathScand,PY22ETDS,BDGW23PLMS,OS23GGD,Ta2304arXiv}. 

In this paper, we explore the cup and cap products 
within the cohomology and homology groups of ample groupoids. 
The cup product, denoted by $\smile$, gives bilinear maps  
\[
H^n(\G,\Z)\times H^m(\G,A)\to H^{n+m}(\G,A)
\]
for $n,m\geq0$ (Theorem \ref{defofcup}), 
and the cap product, denoted by $\frown$, gives bilinear maps  
\[
H_n(\G,\Z)\times H^m(\G,A)\to H_{n-m}(\G,A)
\]
for $n\geq m\geq0$ (Theorem \ref{defofcap}). 
These products are fundamental and important in algebraic topology. 
As is well-known, the cup product provides a way 
to multiply cohomology classes, 
thereby giving the cohomology a graded ring structure. 
Similarly, the cap product is crucial 
for leading to Poincar\'e duality for closed orientable manifolds. 
Introducing the cup and cap products for ample groupoids is important 
as it enhances the invariants 
provided by the cohomology and homology of these groupoids. 
These operations deepen our understanding of 
cohomology and homology groups and 
may be useful for their computation. 
Furthermore, these products could potentially relate to the $K$-groups of 
$C^*$-algebras, suggesting intriguing future research directions. 

We aim to introduce the cup and cap products in a manner 
that is both thorough and accessible. 
Our goal is to ensure that 
even those who may not be familiar with algebraic topology can 
easily grasp these concepts. 
We will carefully explain the definitions and 
discuss the fundamental properties of these products, 
making this material approachable for a broad mathematical audience. 
Next, as one application of the cap product, 
we will show how it naturally arises in the analysis of 
automorphisms of groupoid $C^*$-algebras. 
As observed in \cite[Section 5]{Re_text}, 
a cocycle $\xi:\G\to\T$ (i.e.\ a homomorphism) gives rise to 
an automorphism $\alpha_\xi$ of the groupoid $C^*$-algebra $C^*_r(\G)$, 
and $\alpha_\xi$ is inner if and only if 
the cohomology class of $\xi$ is zero in $H^1(\G,\T)$. 
In many concrete settings, 
the homomorphism $\xi:\G\to\T$ often lifts to a homomorphism $\G\to\R$, 
for example when $\G$ arises an action of $\Z^N$. 
In such a case, $\xi$ is homotopic to the constant map within cocycles, 
which implies that 
$\alpha_\xi$ is homotopic to the identity in $\Aut(C^*_r(\G))$. 
Then, it is natural to ask 
whether the automorphism $\alpha_\xi$ is asymptotically inner. 
Using the cap product, we introduce an invariant for $\xi$ 
which determines asymptotic innerness of $\alpha_\xi$ 
(Proposition \ref{asympinner} and Theorem \ref{asympinnerZ2}). 

This paper is structured as follows. 
In Section 2, we review preliminary concepts related to ample groupoids, 
including definitions and basic properties. 
Section 3 is devoted to the definition and fundamental properties of 
cup and cap products of homology and cohomology groups for ample groupoids. 
Section 4 discusses an application of the cap products 
for automorphisms of groupoid $C^*$-algebras. 
Finally, in Section 5, we present several examples 
to illustrate the applicability and usefulness of these products. 
In particular, we explicitly compute the cup products 
for the cohomology of the Penrose tiling and the Ammann tiling. 
To the best of our knowledge, 
this is the first time the cup products 
for the cohomology of tiling spaces have been computed. 
We believe that these results will open new avenues 
for research in the cohomology of tiling spaces.

\begin{acknowledgements}
We would like to thank the anonymous referee 
for many helpful comments, 
including pointing out prior work by Crainic and Moerdijk 
on the cohomology and homology of \'etale groupoids, 
as well as known issues concerning 
the proof of the gap labelling conjecture. 
We are also grateful to A. Miller and V. Proietti 
for kindly informing us 
about the current status of the gap labelling conjecture.
\end{acknowledgements}

\section{Preliminaries}

\subsection{Ample groupoids}

The cardinality of a set $A$ is written $\#A$ and 
the characteristic function of $A$ is written $1_A$. 
We say that a subset of a topological space is clopen 
if it is both closed and open. 
A topological space is said to be totally disconnected 
if its connected components are singletons. 
By a Cantor set, 
we mean a compact, metrizable, totally disconnected space 
with no isolated points. 
It is known that any two such spaces are homeomorphic. 

In this article, by an \'etale groupoid 
we mean a second countable locally compact Hausdorff groupoid 
such that the range map is a local homeomorphism. 
(We emphasize that our \'etale groupoids are always assumed to be Hausdorff, 
while non-Hausdorff groupoids are also studied actively.) 
We refer the reader to \cite{Re_text,R08Irish} 
for background material on \'etale groupoids. 
For an \'etale groupoid $\G$, 
we let $\G^{(0)}$ denote the unit space and 
let $s$ and $r$ denote the source and range maps, 
i.e.\ $s(g)=g^{-1}g$, $r(g)=gg^{-1}$. 
A subset $U\subset\G$ is called a bisection 
if $r|U$ and $s|U$ are injective. 
The \'etale groupoid $\G$ has a topology with a basis 
consisting of open bisections. 
For $x\in\G^{(0)}$, 
the set $r(s^{-1}(x))$ is called the orbit of $x$. 
When every orbit is dense in $\G^{(0)}$, $\G$ is said to be minimal. 
For a subset $Y\subset\G^{(0)}$, 
the reduction of $\G$ to $Y$ is $r^{-1}(Y)\cap s^{-1}(Y)$ and 
denoted by $\G|Y$. 
If $Y$ is open, then 
the reduction $\G|Y$ is an \'etale subgroupoid of $\G$ in an obvious way. 

An \'etale groupoid $\G$ is called ample 
if its unit space $\G^{(0)}$ is totally disconnected. 
An \'etale groupoid is ample if and only if 
it has a basis for its topology consisting of compact open bisections. 

Typical examples of \'etale groupoids are transformation groupoids. 
Let $\phi:\Gamma\curvearrowright X$ be an action of 
a countable discrete group $\Gamma$ 
on a locally compact Hausdorff space $X$. 
The transformation groupoid $\G$ is $\Gamma\times X$ 
equipped with the product topology. 
The unit space of $\G$ is given by $\G^{(0)}=\{1\}\times X$ 
(where $1$ is the identity of $\Gamma$), 
with range and source maps 
$r(\gamma,x)=(1,\phi_\gamma(x))$ and $s(\gamma,x)=(1,x)$. 
The unit space $\G^{(0)}$ is often identified with $X$. 
Multiplication is given by 
$(\gamma',\phi_\gamma(x))\cdot(\gamma,x)=(\gamma'\gamma,x)$, 
and the inverse of $(\gamma,x)$ is $(\gamma^{-1},\phi_\gamma(x))$. 
Such a transformation groupoid is always \'etale. 
It is ample if and only if $X$ is totally disconnected. 
Moreover, $\G$ is minimal if and only if the action $\phi$ is minimal 
i.e., for all $x\in X$, 
the orbit $\{\phi_\gamma(x)\mid\gamma\in\Gamma\}$ is dense in $X$.

\subsection{Homology groups and cohomology groups}

Let $A$ be a topological abelian group. 
For a locally compact Hausdorff space $X$, 
let $C(X,A)$ be the set of $A$-valued continuous functions. 
We denote by $C_c(X,A)\subset C(X,A)$ 
the subset consisting of functions with compact support. 
With pointwise addition, 
$C(X,A)$ and $C_c(X,A)$ are abelian groups. 

Let $\pi:X\to Y$ be a local homeomorphism 
between locally compact Hausdorff spaces. 
For $f\in C_c(X,A)$, we define a map $\pi_*(f):Y\to A$ by 
\[
\pi_*(f)(y):=\sum_{\pi(x)=y}f(x). 
\]
It is not so hard to see that $\pi_*(f)$ belongs to $C_c(Y,A)$ and 
that $\pi_*$ is a homomorphism from $C_c(X,A)$ to $C_c(Y,A)$. 
Besides, if $\pi':Y\to Z$ is another local homeomorphism to 
a locally compact Hausdorff space $Z$, then 
one can check $(\pi'\circ\pi)_*=\pi'_*\circ\pi_*$ in a direct way. 
Thus, $C_c(\cdot,A)$ is a covariant functor 
from the category of locally compact Hausdorff spaces 
with local homeomorphisms 
to the category of abelian groups with homomorphisms. 

Let $\G$ be an \'etale groupoid. 
For $n\in\N$, we write $\G^{(n)}$ 
for the space of composable strings of $n$ elements in $\G$, that is, 
\[
\G^{(n)}:=\{(g_1,g_2,\dots,g_n)\in\G^n\mid
s(g_i)=r(g_{i+1})\text{ for all }i=1,2,\dots,n{-}1\}. 
\]
For $n\geq2$ and $i=0,1,\dots,n$, 
we let $d_i^{(n)}:\G^{(n)}\to\G^{(n-1)}$ be a map defined by 
\[
d_i^{(n)}(g_1,g_2,\dots,g_n):=\begin{cases}
(g_2,g_3,\dots,g_n) & i=0 \\
(g_1,\dots,g_ig_{i+1},\dots,g_n) & 1\leq i\leq n{-}1 \\
(g_1,g_2,\dots,g_{n-1}) & i=n. 
\end{cases}
\]
When $n=1$, we let $d_0^{(1)},d_1^{(1)}:\G^{(1)}\to\G^{(0)}$ be 
the source map and the range map, respectively. 
Clearly the maps $d_i^{(n)}$ are local homeomorphisms. 
Define the homomorphisms $\partial_n:C_c(\G^{(n)},A)\to C_c(\G^{(n-1)},A)$ 
by 
\[
\partial_n:=\sum_{i=0}^n(-1)^id^{(n)}_{i*}. 
\]
The abelian groups $C_c(\G^{(n)},A)$ 
together with the boundary operators $\partial_n$ form a chain complex. 
When $n=0$, we let $\partial_0:C_c(\G^{(0)},A)\to0$ be the zero map. 

\begin{definition}
[{\cite[Section 3.1]{CM00crelle},\cite[Definition 3.1]{Ma12PLMS}}]
\label{homology}
For $n\geq0$, we let $H_n(\G,A)$ be the homology groups of 
the chain complex above, 
i.e.\ $H_n(\G,A):=\Ker\partial_n/\Im\partial_{n+1}$, 
and call them the homology groups of $\G$ with constant coefficients $A$. 
When $A=\Z$, we simply write $H_n(\G):=H_n(\G,\Z)$. 
In addition, we call elements of $\Ker\partial_n$ cycles and 
elements of $\Im\partial_{n+1}$ boundaries. 
For a cycle $f\in\Ker\partial_n$, 
its equivalence class in $H_n(\G,A)$ is denoted by $[f]$. 
\end{definition}

In \cite{CM00crelle}, 
for any \'etale groupoid $\G$ and any $\G$-sheaf $A$, 
the homology groups $H_n(\G;A)$ are introduced. 
In this article, 
we restrict our attention to the case of constant sheaves. 

Next, we introduce cohomology groups of $\G$. 
For $n\geq0$, 
we define the homomorphisms $\delta^n:C(\G^{(n)},A)\to C(\G^{(n+1)},A)$ by
\[
\delta^n(\xi):=\sum_{i=0}^{n+1}(-1)^i(\xi\circ d^{(n+1)}_i)
\]
for $\xi\in C(G^{(n)},A)$. 
We let $\delta^{-1}:0\to C(\G^{(0)},A)$ be the zero map. 
The abelian groups $C(\G^{(n)},A)$ 
together with the coboundary operators $\delta^n$ form a cochain complex. 

\begin{definition}
[{\cite[Chapter III]{MR660658}, \cite[Definition 1.12]{Re_text}}]
\label{cohomology}
For $n\geq0$, we let $H^n(\G,A)$ be the cohomology groups of 
the cochain complex above, 
i.e.\ $H^n(\G,A):=\Ker\delta^n/\Im\delta^{n-1}$, 
and call them the cohomology groups of $\G$ with constant coefficients $A$. 
When $A=\Z$, we simply write $H^n(\G):=H^n(\G,\Z)$. 
In addition, we call elements of $\Ker\delta^n$ cocycles and 
elements of $\Im\delta^{n-1}$ coboundaries. 
For a cocycle $\xi\in\Ker\delta^n$, 
its equivalence class in $H^n(\G,A)$ is denoted by $[\xi]$. 
\end{definition}

In \cite{MR660658}, 
for any topological groupoid $\G$ and any $\G$-sheaf $A$, 
the cohomology groups $H^n(\G;A)$ are introduced. 
In this article, 
we restrict our attention to the case of 
constant sheaves on \'etale groupoids. 

We will often make use of the following two lemmas. 

\begin{lemma}\label{dd=dd}
Let $n\in\N$. 
For any $0\leq i<j\leq n$, we have 
\[
d_i^{(n-1)}\circ d_j^{(n)}=d_{j-1}^{(n-1)}\circ d_i^{(n)}. 
\]
\end{lemma}

Let $X$ be a locally compact Hausdorff space. 
For $f\in C_c(X,\Z)$ and $\xi\in C(X,A)$, 
we let $f\cdot\xi$ denote the continuous function from $X$ to $A$ 
defined by
\[
(f\cdot\xi)(x):=f(x)\cdot\xi(x)
\]
for all $x\in X$. 
Since $f$ has compact support, so does $f \cdot\xi$. 

\begin{lemma}\label{fcdotxi}
Let $\pi:X\to Y$ be a local homeomorphism 
between locally compact Hausdorff spaces $X$ and $Y$. 
For $f\in C_c(X,\Z)$ and $\xi\in C(Y,A)$, one has 
\[
\pi_*(f)\cdot\xi=\pi_*\left(f\cdot(\xi\circ\pi)\right). 
\]
\end{lemma}

\section{Cup and cap products}

In this section, 
we introduce the cup and cap products for ample groupoids. 
The cup and cap products have already been studied 
by Crainic and Moerdijk \cite{CM9905arXiv,Cr00thesis} 
in the context of sheaves on \'etale groupoids 
(see Remarks 3.3 and 3.6 below). 
Our aim here is to provide a concrete and explicit model 
for these products 
in the setting of constant sheaves on ample groupoids, 
designed to be accessible 
without requiring extensive background in algebraic topology. 

We start by introducing the cup product. 
Let $A$ be a topological abelian group. 

Let $n,m\in\N$.
For $\xi\in C(\G^{(n)},\Z)$ and $\eta\in C(\G^{(m)},A)$, 
we define $\xi\smile\eta\in C(\G^{(n+m)},A)$ by
\[
(\xi\smile\eta)(g_1,g_2,\dots,g_{n+m})
:=\xi(g_1,g_2,\dots,g_n)\cdot\eta(g_{n+1},g_{n+2},\dots,g_{n+m}). 
\]
When $n=0$, we define $\xi\smile\eta\in C(\G^{(m)},A)$ by
\[
(\xi\smile\eta)(g_1,g_2,\dots,g_m)
:=\xi(r(g_1))\cdot\eta(g_1,g_2,\dots,g_m). 
\]
Similarly, when $m=0$, 
we define $\xi\smile\eta\in C(\G^{(n)},A)$ by
\[
(\xi\smile\eta)(g_1,g_2,\dots,g_m)
:=\xi(g_1,g_2,\dots,g_n)\cdot\eta(s(g_n)). 
\]
It is straightforward to verify that the map $\smile$ is bilinear. 
It is also easy to see that $\smile$ is associative, 
i.e. for any $\xi\in C(\G^{(n)},\Z)$, $\eta\in C(\G^{(m)},\Z)$ 
and $\zeta\in C(\G^{(l)},A)$, 
\[
(\xi\smile\eta)\smile\zeta=\xi\smile(\eta\smile\zeta)
\]
holds. 

\begin{proposition}\label{deltaofcup}
In the setting above, one has 
\[
\delta^{n+m}(\xi\smile\eta)
=(\delta^n(\xi)\smile\eta)+(-1)^n(\xi\smile\delta^m(\eta)). 
\]
In particular, 
when $\xi$ and $\eta$ are cocycles, 
$\xi\smile\eta$ is again a cocycle. 
\end{proposition}

\begin{proof}
It follows from the definitions of $\delta^{n+m}$ and $\xi\smile\eta$ 
that 
\begin{align*}
\delta^{n+m}(\xi\smile\eta)
&=\sum_{i=0}^{n+m+1}(-1)^i(\xi\smile\eta)\circ d^{(n+m+1)}_i \\
&=
\sum_{i=0}^n(-1)^i(\xi\smile\eta)\circ d^{(n+m+1)}_i
+\sum_{i=n+1}^{n+m+1}(-1)^i(\xi\smile\eta)\circ d^{(n+m+1)}_i. 
\end{align*}
Thus,
\[
\delta^{n+m}(\xi\smile\eta)
=\sum_{i=0}^n(-1)^i\left((\xi\circ d^{(n+1)}_i)\smile\eta\right)
+\sum_{i=1}^{m+1}(-1)^{n+i}\left(\xi\smile(\eta\circ d^{(m+1)}_i)\right). 
\]
Since 
$(\xi\circ d^{(n+1)}_{n+1})\smile\eta=\xi\smile(\eta\circ d^{(m+1)}_0)$, 
we have 
\[
\delta^{n+m}(\xi\smile\eta)
=\sum_{i=0}^{n+1}(-1)^i\left((\xi\circ d^{(n+1)}_i)\smile\eta\right)
+\sum_{i=0}^{m+1}(-1)^{n+i}\xi\smile\left((\eta\circ d^{(m+1)}_i)\right). 
\]
Therefore, we obtain 
\[
\delta^{n+m}(\xi\smile\eta)
=(\delta^n(\xi)\smile\eta)+(-1)^n(\xi\smile\delta^m(\eta)). 
\]
\end{proof}

\begin{theorem}[cup product]\label{defofcup}
A map $\smile:H^n(\G,\Z)\times H^m(\G,A)\to H^{n+m}(\G,A)$ given by 
\[
([\xi],[\eta])\mapsto [\xi\smile\eta]
\]
is well-defined. 
We write $[\xi]\smile[\eta]:=[\xi\smile\eta]$, 
and call it the cup product of $[\xi]$ and $[\eta]$. 
\end{theorem}

\begin{proof}
Suppose that 
$\xi\in C(\G^{(n)},\Z)$ and $\eta\in C(\G^{(m)},A)$ are cocycles. 
We claim that $\xi\smile\eta$ is a coboundary 
if $\xi$ or $\eta$ is a coboundary. 

First, we assume that $\xi$ is a coboundary. 
There exists $\xi'\in C(\G^{(n-1)},\Z)$ 
such that $\delta^{n-1}(\xi')=\xi$. 
Using Proposition \ref{deltaofcup}, we have 
\[
\delta^{n+m-1}(\xi'\smile\eta)
=(\delta^{n-1}(\xi')\smile\eta)+(-1)^{n-1}(\xi'\smile\delta^m(\eta))
=\xi\smile\eta. 
\]
Hence, $\xi\smile\eta$ is a coboundary.
Second, we assume that $\eta$ is a coboundary. 
There exists $\eta'\in C(\G^{(m-1)},A)$ 
such that $\delta^{m-1}(\eta')=\eta$. 
By Proposition \ref{deltaofcup}, we have 
\[
\delta^{n+m-1}(\xi\smile\eta')
=(\delta^n(\xi)\smile\eta')+(-1)^n(\xi\smile\delta^{m-1}(\eta'))
=(-1)^n(\xi\smile\eta). 
\]
Hence, $\xi\smile\eta$ is a coboundary, 
and the proof of the claim is completed. 

When $\xi$ and $\eta$ are cocycles, by Proposition \ref{deltaofcup}, 
$\xi\smile\eta$ is also a cocycle. 
Furthermore, by the claim above, 
the cohomology class of $\xi\smile\eta$ depends only on 
the cohomology classes of $\xi$ and $\eta$. 
This implies that the cup product is well-defined. 
\end{proof}

The cup product for \'etale groupoids has been studied 
in the literature via sheaf-theoretic approaches. 
To address this and provide clarity, 
we include the following remark. 

\begin{remark}\label{remofcup}
For any \'etale groupoid $\G$ and $\G$-sheaf $A$, 
Moerdijk \cite{Mo98Topology} proved that 
Haefliger's cohomology $H^*(\G;A)$ is 
isomorphic to the sheaf cohomology $H^*(B\G;\tilde A)$, 
where $B\G$ is the classifying space of $\G$ and 
$\tilde A$ is the sheaf on $B\G$ induced by $A$. 
The sheaf cohomology has cup products 
(see \cite[Section II.6.6]{MR345092} for instance), 
and the cup product introduced in the theorem above is its particular case. 
In \cite{Mo98Topology}, 
the isomorphism between $H^*(\G;A)$ and $H^*(B\G;\tilde A)$ 
was constructed via the cohomology of the topological category $\simp(\G)$, 
which is the topological sum of all $\G^{(n)}$, $n\geq0$. 
When $A$ is a constant sheaf on an ample groupoid $\G$, 
one can check that the isomorphisms intertwine cup products. 
Another description of cup products is given 
in \cite[Section 2.10]{CM9905arXiv} 
(which is the preprint version of \cite{CM00crelle}). 
\end{remark}

Next, we introduce the cap product for ample groupoids. 
Let $m,n\in\N$ be such that $m\leq n$. 
For $f\in C_c(\G^{(n)},\Z)$ and $\xi\in C(\G^{(m)},A)$, 
we define $f\frown\xi\in C_c(\G^{(n-m)},A)$ by 
\[
f\frown\xi:=\left(d^{(n-m+1)}_{0*}\circ\cdots\circ d^{(n)}_{0*}\right)
\left(f\cdot(\xi\circ d^{(m+1)}_{m+1}\circ\cdots\circ d^{(n)}_n)\right). 
\]
In other words, for every $(h_1,h_2,\dots,h_{n-m})\in\G^{(n-m)}$, 
\[
(f\frown\xi)(h_1,h_2,\dots,h_{n-m})
=\sum f(g_1,g_2,\dots,g_m,h_1,\dots,h_{n-m})\cdot\xi(g_1,g_2,\dots,g_m), 
\]
where the sum is over all $(g_1,g_2,\dots,g_m)\in\G^{(m)}$ 
satisfying $s(g_m)=r(h_1)$. 
When $m=n$, the definition should be understood as 
\[
f\frown\xi:=\left(d^{(1)}_{0*}\circ\cdots\circ d^{(n)}_{0*}\right)
\left(f\cdot\xi\right), 
\]
that is, 
\[
(f\frown\xi)(x)
=\sum f(g_1,g_2,\dots,g_m)\cdot\xi(g_1,g_2,\dots,g_m)
\]
where the sum is over all $(g_1,g_2,\dots,g_m)\in\G^{(m)}$ 
satisfying $s(g_m)=x$. 

\begin{proposition}\label{partialofcap}
When $m<n$, we have 
\[
\partial_{n-m}(f\frown\xi)
=(-1)^m\left((\partial_n(f)\frown\xi)-(f\frown\delta^m(\xi))\right). 
\]
In particular, 
when $f$ is a cycle and $\xi$ is a cocycle, 
$f\frown\xi$ is a cycle. 
\end{proposition}

\begin{proof}
By Lemma \ref{dd=dd}, 
\begin{align*}
&\partial_{n-m}(f\frown\xi)\\
&=
\sum_{i=0}^{n-m}(-1)^i
\left(d^{(n-m)}_{i*}\circ d^{(n-m+1)}_{0*}\circ\cdots\circ d^{(n)}_{0*}\right)
\left(f\cdot(\xi\circ d^{(m+1)}_{m+1}\circ\cdots\circ d^{(n)}_n)\right)\\
&=
\sum_{i=0}^{n-m}(-1)^i
\left(d^{(n-m)}_{0*}\circ d^{(n-m+1)}_{0*}\circ
\cdots\circ d^{(n)}_{i+m*}\right)
\left(f\cdot(\xi\circ d^{(m+1)}_{m+1}\circ\cdots\circ d^{(n)}_n)\right)\\
&=
(-1)^m\sum_{i=m}^{n}(-1)^i
\left(d^{(n-m)}_{0*}\circ d^{(n-m+1)}_{0*}\circ\cdots\circ d^{(n)}_{i*}\right)
\left(f\cdot(\xi\circ d^{(m+1)}_{m+1}\circ\cdots\circ d^{(n)}_n)\right). 
\end{align*}
When $m+1\leq i\leq n$, 
\[
d^{(m+1)}_{m+1}\circ\cdots\circ d^{(n)}_n
=d^{(m+1)}_{m+1}\circ\cdots\circ d^{(n-1)}_{n-1}\circ d^{(n)}_i. 
\]
This, together with Lemma \ref{fcdotxi}, implies 
\begin{align*}
&\partial_{n-m}(f\frown\xi)\\
&=
\left(d^{(n-m)}_{0*}\circ\cdots\circ d^{(n)}_{m*}\right)
\left(f\cdot(\xi\circ d^{(m+1)}_{m+1}\circ\cdots\circ d^{(n)}_n)\right)\\
&\quad+(-1)^m\sum_{i=m+1}^{n}(-1)^i
\left(d^{(n-m)}_{0*}\circ\cdots\circ d^{(n)}_{i*}\right)
\left(f\cdot(\xi\circ d^{(m+1)}_{m+1}\circ\cdots\circ d^{(n)}_i)\right)\\
&=
\left(d^{(n-m)}_{0*}\circ\cdots\circ d^{(n)}_{m*}\right)
\left(f\cdot(\xi\circ d^{(m+1)}_{m+1}\circ\cdots\circ d^{(n)}_n)\right)\\
&\quad+(-1)^m\sum_{i=m+1}^{n}(-1)^i
\left(d^{(n-m)}_{0*}\circ\cdots\circ d^{(n-1)}_{0*}\right)
\left(d^{(n)}_{i*}(f)\cdot
(\xi\circ d^{(m+1)}_{m+1}\circ\cdots\circ d^{(n-1)}_{n-1})\right). 
\end{align*}
Since 
\[
\partial_n(f)\frown\xi
=\sum_{i=0}^n(-1)^i
\left(d^{(n-m)}_{0*}\circ\cdots\circ d^{(n-1)}_{0*}\right)
\left(d^{(n)}_{i*}(f)\cdot
(\xi\circ d^{(m+1)}_{m+1}\circ\cdots\circ d^{(n-1)}_{n-1})\right), 
\]
we get 
\begin{align*}
&\partial_{n-m}(f\frown\xi)\\
&=
\left(d^{(n-m)}_{0*}\circ\cdots\circ d^{(n)}_{m*}\right)
\left(f\cdot(\xi\circ d^{(m+1)}_{m+1}\circ\cdots\circ d^{(n)}_n)\right)
+(-1)^m(\partial_n(f)\frown\xi)\\
&\quad-(-1)^m\sum_{i=0}^{m}(-1)^i
\left(d^{(n-m)}_{0*}\circ\cdots\circ d^{(n-1)}_{0*}\right)
\left(d^{(n)}_{i*}(f)\cdot
(\xi\circ d^{(m+1)}_{m+1}\circ\cdots\circ d^{(n-1)}_{n-1})\right)\\
&=
\left(d^{(n-m)}_{0*}\circ\cdots\circ d^{(n)}_{m*}\right)
\left(f\cdot(\xi\circ d^{(m+1)}_{m+1}\circ\cdots\circ d^{(n)}_n)\right)
+(-1)^m(\partial_n(f)\frown\xi)\\
&\quad-(-1)^m\sum_{i=0}^{m}(-1)^i
\left(d^{(n-m)}_{0*}\circ\cdots\circ d^{(n)}_{i*}\right)
\left(f\cdot(\xi\circ d^{(m+1)}_{m+1}\circ\cdots\circ d^{(n)}_i)\right), 
\end{align*}
where we have used Lemma \ref{fcdotxi} in the second equality. 
When $0\leq i\leq m$, 
\[
d^{(n-m)}_0\circ d^{(n-m+1)}_0\circ\cdots\circ d^{(n)}_i
=d^{(n-m)}_0\circ d^{(n-m+1)}_0\circ\cdots\circ d^{(n)}_0. 
\]
This, together with Lemma \ref{dd=dd}, implies 
\begin{align*}
&\partial_{n-m}(f\frown\xi)\\
&=(-1)^m(\partial_n(f)\frown\xi)+
\left(d^{(n-m)}_{0*}\circ\cdots\circ d^{(n)}_{0*}\right)
\left(f\cdot(\xi\circ d^{(m+1)}_{m+1}\circ\cdots\circ d^{(n)}_n)\right)\\
&\quad-(-1)^m\sum_{i=0}^{m}(-1)^i
\left(d^{(n-m)}_{0*}\circ\cdots\circ d^{(n)}_{0*}\right)
\left(f\cdot
(\xi\circ d^{(m+1)}_i\circ d^{(m+2)}_{m+2}\circ\cdots\circ d^{(n)}_n)\right)\\
&=(-1)^m(\partial_n(f)\frown\xi)\\
&\quad-(-1)^m\sum_{i=0}^{m+1}(-1)^i
\left(d^{(n-m)}_{0*}\circ\cdots\circ d^{(n)}_{0*}\right)
\left(f\cdot
(\xi\circ d^{(m+1)}_i\circ d^{(m+2)}_{m+2}\circ\cdots\circ d^{(n)}_n)\right)\\
&=(-1)^m\left(\partial_n(f)\frown\xi)-f\frown\delta^m(\xi)\right), 
\end{align*}
which completes the proof. 
\end{proof}

\begin{theorem}[cap product]\label{defofcap}
Let $m,n\in\N\cup\{0\}$ be such that $m\leq n$.
A map $\frown:H_n(\G,\Z)\times H^m(\G,A)\to H_{n-m}(\G,A)$ given by 
\[
([f],[\xi])\mapsto [f\frown\xi]
\]
is well-defined. 
We write $[f]\frown[\xi]:=[f\frown\xi]$, 
and call it the cap product of $[f]$ and $[\xi]$. 
\end{theorem}

\begin{proof}
Suppose that 
$f\in C_c(\G^{(n)},\Z)$ is a cycle and 
$\xi\in C(\G^{(m)},A)$ is a cocycle. 
When $m=n$, $f\frown\xi\in C_c(\G^{(0)},A)$ is clearly a cycle. 
When $m<n$, by Proposition \ref{partialofcap}, 
$f\frown\xi\in C_c(\G^{(n-m)},A)$ is a cycle. 
We claim that $f\frown\xi$ is a boundary 
if $f$ is a boundary or $\xi$ is a coboundary. 

First, we assume that $f$ is a boundary. 
There exists $f'\in C_c(\G^{(n+1)},\Z)$ 
such that $\partial_{n+1}(f')=f$. 
By Proposition \ref{partialofcap}, we have 
\[
\partial_{n-m+1}(f'\frown\xi)
=(-1)^m\left((\partial_{n+1}(f')\frown\xi)
+(f'\frown\delta^m(\xi))\right)
=(-1)^m(f\frown\xi). 
\]
Hence, $f\frown\xi$ is a boundary.
Second, we assume that $\xi$ is a coboundary. 
There exists $\xi'\in C(\G^{(m-1)},A)$ 
such that $\delta^{m-1}(\xi')=\xi$. 
By Proposition \ref{partialofcap}, we have 
\[
\partial_{n-m+1}(f\frown\xi')
=(-1)^{m-1}\left((\partial_n(f)\frown\xi')
+(f\frown\delta^{m-1}(\xi'))\right)
=(-1)^{m-1}(f\frown\xi). 
\]
Hence, $f\frown\xi$ is a boundary, 
and the proof of the claim is completed. 

It follows that 
the homology class of $f\frown\xi$ depends only on 
the homology and cohomology classes of $f$ and $\xi$. 
This implies that the cap product is well-defined. 
\end{proof}

As with the cup product, 
the cap product has a well-established abstract formulation. 
We therefore include the following remark 
to connect our concrete construction with previous works. 

\begin{remark}\label{remofcap}
A model of cap product has been constructed 
in Crainic's PhD thesis (\cite[Section 2.2.12]{Cr00thesis}) 
in the following way. 
Let $\G$ be an \'etale groupoid. 
For any $\G$-sheaves $A,B$, 
one can construct $\Ext_\G^*(A,B)$ 
as the right derived functor of $B\mapsto\Hom_\G(A,B)$. 
The particular case is $\Ext_\G^*(\Z,B)=H^*(\G;B)$. 
As observed in \cite[Section 1.2.6]{Cr00thesis}, 
the Ext groups act on the the homology: 
\[
H_n(\G;A)\times\Ext_\G^m(A,B)\to H_{n-m}(\G;B). 
\]
When working over $\Z$, 
this yields the cap product for \'etale groupoids: 
\[
\frown\ :H_n(\G;\Z)\times H^m(\G;A)\to H_{n-m}(\G;A). 
\]
Theorem \ref{defofcap} gives a concrete description of the cap product 
in the case of constant sheaves on ample groupoids. 
\end{remark}

The following relationship holds between the cup and cap products. 
We remark that 
this also follows from the Crainic's presentation of cap products 
mentioned in the remark above. 
See \cite[Section 1.2.6, 2.2.12]{Cr00thesis}. 

\begin{proposition}
Let $n,m,l\in\N$ be such that $m+l\leq n$. 
For any $f\in C_c(G^{(n)},\Z)$, $\xi\in C(G^{(m)},\Z)$ and 
$\eta\in C(G^{(l)},A)$, we have 
\[
f\frown(\xi\smile\eta)=(f\frown\xi)\frown\eta. 
\]
\end{proposition}

\begin{proof}
For $(h_1,h_2,\dots,h_{n-m-l})\in\G^{(n-m-l)}$, 
we have 
\begin{align*}
&\left(f\frown(\xi\smile\eta)\right)(h_1,h_2,\dots,h_{n-m-l})\\
&=\sum f(g_1,\dots,g_{m+l},h_1,\dots,h_{n-m-l})
\cdot(\xi\smile\eta)(g_1,\dots,g_{m+l})\\
&=\sum f(g_1,\dots,g_{m+l},h_1,\dots,h_{n-m-l})
\cdot\xi(g_1,\dots,g_m)\cdot\eta(g_{m+1},\dots,g_{m+l}), 
\end{align*}
where the sum is over all $(g_1,\dots,g_{m+l})\in\G^{(m+l)}$ 
satisfying $s(g_{m+l})=r(h_1)$. 
On the other hand, we have 
\begin{align*}
&\left((f\frown\xi)\frown\eta\right)(h_1,h_2,\dots,h_{n-m-l})\\
&=\sum(f\frown\xi)(p_1,\dots,p_l,h_1,\dots,h_{n-m-l})
\cdot\eta(p_1,\dots,p_l)\\
&=\sum\sum f(q_1,\dots,q_m,p_1,\dots,p_l,h_1,\dots,h_{n-m-l})
\cdot\xi(q_1,\dots,q_m)\cdot\eta(p_1,\dots,p_l), 
\end{align*}
where the first sum is over $(p_1,\dots,p_l)\in\G^{(l)}$ 
satisfying $s(p_l)=r(h_1)$ and 
the second sum is over $(q_1,\dots,q_m)\in\G^{(m)}$ 
satisfying $s(q_m)=r(p_1)$. 
Hence we get the conclusion. 
\end{proof}

In the rest of this section, 
we discuss functoriality of the products. 

Let $\pi:\G\to\H$ be a homomorphism between ample groupoids. 
We let $\pi^{(0)}:\G^{(0)}\to\H^{(0)}$ denote 
the restriction of $\pi$ to $\G^{(0)}$, 
and $\pi^{(n)}:\G^{(n)}\to\H^{(n)}$ denote 
the restriction of the $n$-fold product $\pi\times\pi\times\dots\times\pi$ 
to $\G^{(n)}$. 
It is plain that 
$\pi^{(n)}$ commute with the maps $d^{(n)}_i$. 
Hence we obtain homomorphisms 
\[
H^n(\pi):H^n(\H,A)\to H^n(\G,A). 
\]
Moreover, 
when the homomorphism $\pi$ is a local homeomorphism 
(i.e. an \'etale map), 
we can consider homomorphisms 
$\pi^{(n)}_*:C_c(\G^{(n)},A)\to C_c(\H^{(n)},A)$. 
Since they commute with the boundary operators $\partial_n$, 
we obtain homomorphisms 
\[
H_n(\pi):H_n(\G,A)\to H_n(\H,A). 
\]
It is straightforward to verify the following. 

\begin{proposition}
Let $\pi:\G\to\H$ be a homomorphism between ample groupoids. 
For any $[\xi]\in H^n(\H,\Z)$ and $[\eta]\in H^n(\H,A)$, 
one has 
\[
H^n(\pi)([\xi])\smile H^m(\pi)([\eta])
=H^{n+m}(\pi)([\xi]\smile[\eta]). 
\]
\end{proposition}

As for the cap product, we obtain the following. 

\begin{proposition}
Let $\pi:\G\to\H$ be a homomorphism between ample groupoids. 
Suppose that $\pi$ is a local homeomorphism. 
Let $m,n\in\N$ be such that $m\leq n$. 
Then, for any $[f]\in H_n(\G,\Z)$ and $[\xi]\in H^m(\H,A)$, 
one has 
\[
H_n(\pi)([f])\frown[\xi]
=H_{n-m}(\pi)\left([f]\frown H^m(\pi)([\xi])\right). 
\]
\end{proposition}

The following diagram illustrates the equality above. 
\[
\xymatrix@M=8pt{
H_n(\G,\Z)\times H^m(\G,A) \ar[r]^-\frown 
\ar@<-6ex>[d]_-{H_n(\pi)} & 
H_{n-m}(\G,A) \ar[d]^-{H_{n-m}(\pi)} \\
H_n(\H,\Z)\times H^m(\H,A) \ar[r]^-\frown 
\ar@<-6ex>[u]_-{H^m(\pi)} & 
H_{n-m}(\H,A) \\
}
\]

\begin{proof}
By definition, for every $(h_1,h_2,\dots,h_{n-m})\in\H^{(n-m)}$, 
\begin{align*}
& \left(\pi^{(n)}_*(f)\frown\xi\right)(h_1,h_2,\dots,h_{n-m}) \\
&=\sum\pi^{(n)}_*(f)(g_1,g_2,\dots,g_m,h_1,\dots,h_{n-m})
\cdot\xi(g_1,g_2,\dots,g_m), 
\end{align*}
where the sum is over all $(g_1,g_2,\dots,g_m)\in\H^{(m)}$ 
satisfying $s(g_m)=r(h_1)$. 
By the definition of $\pi^{(n)}_*$, we get 
\begin{align}
& \left(\pi^{(n)}_*(f)\frown\xi\right)(h_1,h_2,\dots,h_{n-m}) \notag \\
&=\sum f(g'_1,g'_2,\dots,g'_m,h'_1,\dots,h'_{n-m})
\cdot\xi(\pi(g'_1),\pi(g'_2),\dots,\pi(g'_m)) \notag \\
&=\sum f(g'_1,g'_2,\dots,g'_m,h'_1,\dots,h'_{n-m})
\cdot(\xi\circ\pi^{(m)})(g'_1,g'_2,\dots,g'_m), \label{lhs}
\end{align}
where the sum is over 
all $(g'_1,\dots,g'_m,h'_1,\dots,h'_{n-m})\in\G^{(n)}$ 
satisfying $\pi^{(n-m)}(h'_1,\dots,h'_{n-m})=(h_1,\dots,h_{n-m})$. 
On the other hand, 
\begin{align}
& \pi^{(n-m)}_*\left(f\frown(\xi\circ\pi^{(m)}))\right)
(h_1,h_2,\dots,h_{n-m}) \notag \\
&=\sum\left(f\frown(\xi\circ\pi^{(m)}))\right)
(h'_1,h'_2,\dots,h'_{n-m}) \notag \\
&=\sum f(g'_1,g'_2,\dots,g'_m,h'_1,\dots,h'_{n-m})
\cdot(\xi\circ\pi^{(m)})(g'_1,g'_2,\dots,g'_m), \label{rhs}
\end{align}
where the sum is over all $(g'_1,\dots,g'_m)\in\G^{(m)}$ 
and $(h'_1,\dots,h'_{n-m})\in\G^{(n-m)}$ 
satisfying $s(g'_m)=r(h'_1)$ and 
$\pi^{(n-m)}(h'_1,\dots,h'_{n-m})=(h_1,\dots,h_{n-m})$. 
By \eqref{lhs} and \eqref{rhs}, we obtain the desired equality. 
\end{proof}

\section{Automorphisms of $C^*$-algebras}

This section discusses 
how cap products serve as invariants of automorphisms of $C^*$-algebras.
In general, an automorphism $\alpha$ of a unital $C^*$-algebra $A$ 
is said to be asymptotically inner, 
if there exists a continuous family of unitaries 
$(u_t)_{t\in[0,\infty)}$ in $A$ such that 
\[
\alpha(a)=\lim_{t\to\infty}u_tau_t^*
\]
holds for all $a\in A$. 
It is an important problem to determine 
if a given automorphism is asymptotically inner or not, 
and a number of results are found in the literature 
(see for example \cite{Ph00DocM,KK01JOP,Li09AJM}). 
Here, we will use the following theorem of H. Lin. 

\begin{theorem}[{\cite[Corollary 11.2]{Li09AJM}}]\label{Lin}
Let $A$ be a unital separable simple $C^*$-algebra 
with tracial rank zero and satisfying the UCT. 
For $\alpha\in\Aut(A)$, the following are equivalent. 
\begin{enumerate}
\item $\alpha$ is asymptotically inner. 
\item $KK(\alpha)=KK(\id)$ and 
the rotation map induced by $\alpha$ is zero. 
\end{enumerate}
\end{theorem}

Throughout this section, we let $\G$ denote an ample groupoid 
arising from a free minimal action of $\Z^N$ on the Cantor set. 
Set $\T:=\{z\in\C\mid\lvert z\rvert=1\}$. 

Let $A:=C^*_r(\G)$ be the reduced groupoid $C^*$-algebra of $\G$. 
When $\xi\in C(\G,\T)$ is a cocycle, 
one can define an automorphism $\alpha_\xi\in\Aut(A)$ by 
\[
\alpha_\xi(a)(g):=\xi(g)a(g)
\]
for $a\in A$ and $g\in\G$. 
Clearly $\alpha_\xi$ is inner 
if and only if $[\xi]=0$ in $H^1(\G,\T)$ 
(see \cite[Proposition 5.7]{Ma12PLMS}). 
We examine conditions under which $\alpha_\xi$ is asymptotically inner. 

We let $M(\G)$ denote 
the set of all $\G$-invariant probability measures on $\G^{(0)}$ 
endowed with the weak-$*$ topology. 
As $\G$ is the ample groupoid of a free action, 
$M(\G)$ can be identified with the space of tracial states $T(A)$ on $A$. 
Let $\Aff(M(\G))$ be 
the space of continuous affine functions $M(\G)\to\R$. 
There exists a natural homomorphism $D_\G:H_0(\G)\to\Aff(M(\G))$ 
defined by 
\[
D_\G([f])(\mu):=\int_{\G^{(0)}}f\ d\mu
\]
for $[f]\in H_0(\G)$ and $\mu\in M(\G)$. 
Define a homomorphism 
\[
\rho_\G:H_0(\G,\T)\to\Aff(M(\G))/\Im D_\G
\]
as follows. 
For $h\in C(\G^{(0)},\T)$, 
we pick $f\in C(\G^{(0)},\R)$ such that $h(x)=e^{2\pi if(x)}$, 
and set 
\[
\rho_\G([h])(\mu):=\int_{\G^{(0)}}f\ d\mu+\Im D_\G. 
\]
It is not so hard to see that 
$\rho_\G$ is a well-defined homomorphism. 

\begin{proposition}\label{asympinner}
Let $N\leq3$. 
If $\alpha_\xi$ is asymptotically inner, then 
\[
\rho_\G\left(c\frown[\xi]\right)=0
\]
for any $c\in H_1(\G)$. 
\end{proposition}

\begin{proof}
We write 
\[
\Delta_\tau:U_\infty(A)_0\to\R/D_A(K_0(A))(\tau)
\]
for the de la Harpe-Skandalis determinant \cite{MR0743629} 
associated with a tracial state $\tau\in T(A)$ on $A$, 
where $D_A:K_0(A)\to\Aff(T(A))$ is the dimension map 
defined by $D_A([p])(\tau)=\tau(p)$. 
Let us recall the definition of $\Delta_\tau$ briefly. 
Suppose that we are given a unitary $u\in U_n(A)_0$. 
There exists a piecewise smooth path $(u_t)_{t\in[0,1]}$ 
in $U_n(A)_0$ such that $u_0=1$, $u_1=u$. 
Then, the de la Harpe-Skandalis determinant of $u$ is 
\[
\Delta_\tau(u)=\frac{1}{2\pi i}
\int_0^1\tau\left(\frac{d\,u_t}{dt}u_t^*\right)\ dt
+D_A(K_0(A))(\tau). 
\]
It is well-known that $\Delta_A(u)(\tau):=\Delta_\tau(u)$ 
gives a homomorphism 
\[
\Delta_A:U_\infty(A)_0\to\Aff(T(A))/\Im D_A, 
\]
see \cite{Ma11JFA} for instance. 
By \cite[Theorem 4.3]{Li09AJM} and its proof, 
when $\alpha_\xi$ is asymptotically inner, 
we have $\Delta_A(u^*\alpha_\xi(u))=0$ 
for all $u\in U_\infty(A)$. 

The groupoid $\G$ arises from a free action of $\Z^N$, 
and so $\G$ is almost finite by \cite[Lemma 6.3]{Ma12PLMS}. 
Therefore, by \cite[Theorem 7.5]{Ma12PLMS}, 
any element $c$ of $H_1(\G)$ is represented 
by a cycle of the form $1_U\in C_c(\G,\Z)$, 
where $U$ is a bisection such that $r(U)=\G^{(0)}=s(U)$. 
We may regard $1_U$ as a unitary in $A=C^*_r(\G)$. 
By a straightforward computation, we see that 
$h:=(1_U)^*\alpha_\xi(1_U)$ is a continuous function on $\G^{(0)}$ 
such that $h(s(g))=\xi(g)$ for all $g\in U$. 
Pick $f\in C(\G^{(0)},\R)$ such that $h(x)=e^{2\pi if(x)}$, 
and put $h_t(x):=e^{2\pi itf(x)}$ for $t\in[0,1]$. 
Then, $(h_t)_{t\in[0,1]}$ is a smooth path  of unitaries 
satisfying $h_0=1$ and $h_1=h$. 
As 
\[
\frac{d\,h_t}{dt}h_t^*=2\pi ifh_th^*_t=2\pi if, 
\]
$\Delta_\tau(h)=\tau(f)+D_A(K_0(A))(\tau)$ holds for every $\tau\in T(A)$. 
It follows from $\Delta_A(h)=0$ that we can conclude 
the map $\tau\mapsto\tau(f)$ belongs to $\Im D_A$. 

Now, let us consider $\rho_\G\left([1_U]\frown[\xi]\right)$. 
By the definition of the cap product, 
\[
(1_U\frown\xi)(x)=\sum_{g\in s^{-1}(x)}1_U(g)\cdot\xi(g)
\]
for all $x\in\G^{(0)}$. 
Thus, $1_U\frown\xi=h$, 
and $\rho_\G\left([1_U]\frown[\xi]\right)$ is equal to 
the equivalence class of the map 
\[
\mu\mapsto\int_{\G^{(0)}}f\ d\mu
\]
in $\Aff(M(\G))/\Im D_\G$. 
Since $\G$ is the transformation groupoid of 
a free minimal action of $\Z^N$ with $N\leq3$, 
under the identification of $T(A)$ with $M(\G)$, the gap labeling theorem 
\cite{BKL01CRASParis,ADRS21JGP} (see also the remark below) 
implies $\Im D_A=\Im D_\G$. 
Hence we obtain $\rho_\G\left([1_U]\frown[\xi]\right)=0$. 
\end{proof}

\begin{remark}
Let $\G$ be an ample groupoid 
arising from a free action of $\Z^N$ on the Cantor set 
and let $A:=C^*_r(\G)$. 
The gap labelling conjecture asks whether or not the image of the map 
\[
D_\G:H_0(\G)\to\Aff(M(\G))
\]
and the image of the map 
\[
D_A:K_0(A)\to\Aff(T(A))
\]
agree under the identification of $T(A)$ with $M(\G)$. 
Several proofs for this conjecture were proposed 
(\cite{BBG06CMP,MR2018220,KP03Michigan}), 
but it is known that they have issues 
(see page 1663 of \cite{GHK13AGT} for instance). 
This is why we need to assume $N\leq3$ in the proposition above. 
\end{remark}

\begin{remark}
Given a cocycle $\xi\in C(\G,\T)$, 
one can consider a $\G$-sheaf $A$ of 
germs of continuous $\Z\times\T$-valued functions 
such that the $\G$-action is given by 
the translation by $\xi$ on the $\T$-component. 
More precisely, for any compact open bisection $U\subset\G$, 
the $\G$-action is described as follows: 
\[
\left((f_1,f_2)\cdot g\right)(r(g))
=\left(f_1(s(g)),f_2(s(g))+f_1(s(g))\xi(g)\right)
\]
for $f_1\in C(s(U),\Z)$, $f_2\in C(s(U),\T)$ and $g\in U$. 
By \cite[Section 3.6]{CM00crelle}, 
we obtain the long exact sequence in homology: 
\[
\xymatrix@M=8pt{
\cdots \ar[r] & 
H_{n+1}(\G,\Z) \ar[r] & 
H_n(\G,\T) \ar[r] & 
H_n(\G;A) \ar[r] & 
H_n(\G,\Z) \ar[r] & \cdots .
}
\]
The cap product $c\mapsto c\frown[\xi]$ appeared 
in the proof of Proposition \ref{asympinner} equals 
the homomorphism $H_1(\G,\Z)\to H_0(\G,\T)$ in this long exact sequence. 
\end{remark}

When $\G$ comes from an action of $\Z^2$ and $H_1(\G)$ is free, 
the converse of the proposition above is also true. 

\begin{theorem}\label{asympinnerZ2}
Suppose that $\G$ is an ample groupoid 
arising from a free minimal action of $\Z^2$ on the Cantor set. 
Let the notation be as above. 
Assume further that $H_1(\G)$ is free. 
For a cocycle $\xi:\G\to\T$, the following are equivalent. 
\begin{enumerate}
\item $\alpha_\xi$ is asymptotically inner. 
\item $\rho_\G\left(c\frown[\xi]\right)=0$ for any $c\in H_1(\G)$. 
\end{enumerate}
\end{theorem}

\begin{proof}
The implication (1)$\implies$(2) is already shown 
in Proposition \ref{asympinner}. 
We prove (2)$\implies$(1). 
By \cite[Theorem B]{Ni24CJM} (see also \cite{KN2107arXiv}), 
$A$ is a unital separable simple $C^*$-algebra 
with tracial rank zero and satisfying the UCT. 
We would like to apply Theorem \ref{Lin} 
(see also \cite[Theorem 3.1]{KK01JOP} and \cite[Theorem 10.7]{Li09AJM}). 

Since $\Z^2$ is free, we can find a cocycle $\eta:\G\to\R$ 
such that $\xi(g)=e^{2\pi i\eta(g)}$ for all $g\in\G$. 
For each $t\in[0,1]$, 
we define $\xi_t:\G\to\T$ by $\xi_t(g)=e^{2\pi it\eta(g)}$. 
Then $(\xi_t)_t$ is a continuous path of cocycles 
such that $\xi_0=1$ and $\xi_1=\xi$. 
It follows that 
$(\alpha_{\xi_t})_t$ is a continuous path in $\Aut(A)$ 
such that $\alpha_{\xi_0}=\id$ and $\alpha_{\xi_1}=\alpha_\xi$. 
In particular, we have $KK(\alpha_\xi)=KK(\id)$. 

It remains to show that the rotation map vanishes 
(\cite[Section 3]{Li09AJM}). 
Consider the following short exact sequence of $C^*$-algebras: 
\[
\xymatrix@M=8pt{
0 \ar[r] & SA \ar[r]^{\iota} & M \ar[r]^{\pi} & A \ar[r] & 0,
}
\]
where 
\[
M:=\left\{x\in C([0,1],A)\mid\alpha_\xi(x(0))=x(1)\right\}
\]
is the mapping torus of $\alpha_\xi$. 
We define a homomorphism $\beta:A\to M$ 
by $\beta(a)(t):=\alpha_{\xi_t}(a)$ for $t\in[0,1]$. 
Clearly $\beta$ is a right inverse of $\pi$. 
Let $R:K_1(M)\to\Aff(T(A))$ be the rotation map 
defined in \cite[Definition 3.2]{Li09AJM}. 
By \cite[Lemma 3.3]{Li09AJM}, we have the following diagram: 
\[
\xymatrix@M=8pt{
0 \ar[r] & K_0(A) \ar[r]^{\iota_*} \ar[dr]_{D_A} & 
K_1(M) \ar@<0.5ex>[r]^{\pi_*} \ar[d]^{R} & 
K_1(A) \ar[r] \ar@<0.5ex>[l]^{\beta_*} & 0 \\
& & \Aff(T(A)) & &
}
\]
We would like to see that 
$\Im(R\circ\beta_*)$ is contained in $\Im D_A$. 
Take a bisection $U\subset\G$ such that $r(U)=\G^{(0)}=s(U)$. 
As 
\[
\frac{d}{dt}\alpha_{\xi_t}(1_U)(g)=2\pi i\eta(g)1_U(g), 
\]
letting $f:\G^{(0)}\to\R$ be 
such that $f(s(g))=\eta(g)$ for all $g\in U$, 
one has 
\[
R([\beta(1_U)])(\tau)
=\frac{1}{2\pi i}\int_0^1\tau(2\pi if)\ dt=\tau(f)
\]
for each $\tau\in T(A)$. 
In the same way as Proposition \ref{asympinner}, 
we can see $1_U\frown\xi=e^{2\pi if}$, and so 
\[
0=\rho_\G([1_U]\frown[\xi])=\tau(f)+\Im D_A
=R([\beta(1_U)])+\Im D_A. 
\]
Therefore $R([\beta(1_U)])$ is contained in $\Im D_A$. 
Since $\G$ is the transformation groupoid of a $\Z^2$-action, 
$K_1(A)$ is canonically isomorphic to $H_1(\G)$. 
Moreover, 
as explained in the proof of Proposition \ref{asympinner}, 
every element of $H_1(\G)$ can be represented of the form $[1_U]$. 
This proves that 
the image of $R\circ\beta_*$ is contained in $\Im D_A$. 
Hence we may find a homomorphism $\theta:K_1(A)\to K_1(M)$ 
such that $\pi_*\circ\theta=\id$ and $R\circ\theta=0$, 
because $K_1(A)\cong H_1(\G)$ is free. 
According to Theorem \ref{Lin}, 
one can conclude that $\alpha_\xi$ is asymptotically inner. 
\end{proof}

\begin{remark}
When $\G$ is the transformation groupoid of a minimal $\Z$-action, 
the theorem above has been shown in \cite[Proposition 6.1]{Ma01CJM}. 
In this case $H_1(\G)$ is $\Z$. 
\end{remark}

\section{Examples}

\subsection{SFT groupoids}

In this subsection, 
we compute a cap product $\frown\ :H_1\times H^1\to H_0$ 
in the setting of SFT groupoids. 
The reader may refer to \cite{Ma12PLMS,Ma15crelle} 
for a more complete background of SFT groupoids. 

Let $(V,E)$ be a finite directed graph, 
where $V$ is a finite set of vertices 
and $E$ is a finite set of edges. 
For $e\in E$, $i(e)$ denotes the initial vertex of $e$ and 
$t(e)$ denotes the terminal vertex of $e$. 
Let $A:=(A(v,w))_{v,w\in V}$ be 
the adjacency matrix of $(V,E)$, that is, 
\[
A(v,w):=\#\{e\in E\mid i(e)=v,\ t(e)=w\}. 
\]
We assume that $A$ is irreducible 
(i.e.\ for all $v,w\in V$ 
there exists $n\in\N$ such that $A^n(v,w)>0$) 
and that $A$ is not a permutation matrix. 
Define 
\[
X:=\{(x_k)_{k\in\N}\in E^\N
\mid t(x_k)=i(x_{k+1})\quad\forall k\in\N\}. 
\]
With the product topology, $X$ is a Cantor set. 
Define a surjective continuous map $\sigma:X\to X$ by 
\[
\sigma(x)_k:=x_{k+1}\quad k\in\N,\ x=(x_k)_k\in X. 
\]
In other words, $\sigma$ is the (one-sided) shift on $X$. 
Observe that $\sigma$ is a local homeomorphism. 
The dynamical system $(X,\sigma)$ is called 
the one-sided irreducible shift of finite type (SFT) 
associated with the graph $(V,E)$ (or the matrix $A$). 

The ample groupoid $\G_{(V,E)}$ is given by 
\[
\G_{(V,E)}:=\{(x,n,y)\in X\times\Z\times X\mid
\exists k,l\in\N,\ n=k{-}l,\ \sigma^k(x)=\sigma^l(y)\}. 
\]
The topology of $\G_{(V,E)}$ is generated by the sets 
\[
\{(x,k{-}l,y)\in\G_{(V,E)}\mid x\in P,\ y\in Q,\ \sigma^k(x)=\sigma^l(y)\}, 
\]
where $P,Q\subset X$ are open and $k,l\in\N$. 
Two elements $(x,n,y)$ and $(x',n',y')$ in $\G_{(V,E)}$ are composable 
if and only if $y=x'$, and the multiplication and the inverse are 
\[
(x,n,y)\cdot(y,n',y')=(x,n{+}n',y'),\quad (x,n,y)^{-1}=(y,-n,x). 
\]
We identify $X$ with the unit space $\G_{(V,E)}^{(0)}$ 
via $x\mapsto(x,0,x)$. 
We call $\G_{(V,E)}$ the SFT groupoid 
associated with the graph $(V,E)$. 

The homology groups of the SFT groupoid $\G_{(V,E)}$ was 
computed in \cite{Ma12PLMS}. 

\begin{theorem}[{\cite[Theorem 4.14]{Ma12PLMS}}]
Let $(V,E)$, $A$ and $\G_{(V,E)}$ be as above. 
Then one has 
\[
H_n(\G_{(V,E)})\cong\begin{cases}\Coker(I-A^t)&n=0\\
\Ker(I-A^t)&n=1\\0&n\geq2, \end{cases}
\]
where the matrix $A$ acts on the abelian group $\Z^V$ by multiplication. 
\end{theorem}

Now, the map $\xi:\G_{(V,E)}\to\Z$ defined by $\xi(x,n,y):=n$ 
is a homomorphism, 
and so it gives rise to the cohomology class $[\xi]\in H^1(\G_{(V,E)})$. 
We would like to compute the map 
\[
\cdot\frown[\xi]\ :H_1(\G_{(V,E)})\to H_0(\G_{(V,E)}). 
\]
Suppose that $a=(a(v))_v\in\Z^V$ belongs to $\Ker(I-A^t)$, i.e. 
\begin{equation}
a(v)=\sum_{w\in V}A^t(v,w)a(w)=\sum_{w\in V}A(w,v)a(w)
=\sum_{t(e)=v}a(i(e)) \label{ainKer}
\end{equation}
for every $v\in V$. 
Define $f\in C_c(\G_{(V,E)},\Z)$ by 
\[
f(g):=\begin{cases}a(i(x_1))&\text{$g=(x,1,\sigma(x))$ for some $x\in X$}\\
0&\text{else.}\end{cases}
\]
One has 
\[
r_*(f)(x)=a(i(x_1))
\]
and 
\[
s_*(f)(x)=\sum_{y\in\sigma^{-1}(x)}f(y,1,x)
=\sum_{t(e)=i(x_1)}a(i(e))
\]
for every $x\in X$. 
Hence, by \eqref{ainKer}, $f$ is a cycle, 
and $[f]\in H_1(\G_{(V,E)})$ corresponds to $a\in\Ker(I-A^t)$. 

It is easy to see 
\[
f\frown\xi=s_*(f), 
\]
and so the homology class $[f]\frown[\xi]$ is equal to $[r_*(f)]$. 
As seen above, $r_*(f)(x)=a(i(x_1))$. 
Therefore, 
$[r_*(f)]\in H_0(\G_{(V,E)})$ corresponds to $a\in\Z^V$. 
Consequently, we get the following theorem. 

\begin{theorem}
Let $\G_{(V,E)}$ be the SFT groupoid as above and 
let $\xi:\G_{(V,E)}\to\Z$ be the homomorphism 
defined by $\xi(x,n,y):=n$. 
Then the map 
\[
\cdot\frown[\xi]\ :H_1(\G_{(V,E)})\to H_0(\G_{(V,E)})
\]
is equal to 
\[
\Ker(I-A^t)\ni a\mapsto a+\Im(I-A^t)\in\Coker(I-A^t). 
\]
\end{theorem}

\subsection{$\Z^N$ actions}

In this subsection, we deal with actions of $\Z^N$. 
We denote the canonical basis of $\Z^N$ by $e_1,e_2,\dots,e_N$. 
Let $\phi:\Z^N\curvearrowright X$ be an action 
on a Cantor set $X$ 
and let $\G:=\Z^N\times X$ be the transformation groupoid 
(see Section 2.1). 

Let $E_i:=\{(e_i,x)\in\G\mid x\in X\}$. 
For every $\sigma\in\mathcal{S}_N$, we put 
\[
E_\sigma:=\left(E_{\sigma(1)}\times E_{\sigma(2)}
\times\dots\times E_{\sigma(N)}\right)
\cap\G^{(N)}, 
\]
and define $f\in C_c(\G^{(N)},\Z)$ by 
\[
f:=\sum_{\sigma\in\mathcal{S}_N}\sgn(\sigma)1_{E_\sigma}, 
\]
where $\sgn(\sigma)\in\{1,-1\}$ is 
the signature of the permutation $\sigma$. 
Then $f$ is a cycle, 
and it is well-known that $[f]$ induces isomorphisms 
\[
[f]\frown\cdot\ :H^i(\G)\to H_{N-i}(\G)
\]
for all $i=0,1,\dots,N$ (Poincar\'e duality). 

For $\xi\in C(\G,\Z)$ and $i=1,2,\dots,N$, 
we define $\xi_i\in C(X,\Z)$ by $\xi_i(x):=\xi(e_i,x)$. 
If $\xi$ is a cocycle, then 
\[
\xi_i+\xi_j\circ\phi_{e_i}=\xi_j+\xi_i\circ\phi_{e_j}
\]
holds for all $i,j=1,2,\dots,N$. 

Suppose that 
$\xi^{(1)},\xi^{(2)},\dots,\xi^{(N)}\in C(\G,\Z)$ are cocycles. 
We would like to describe the homology class 
\[
[f]\frown[\xi^{(1)}]\smile[\xi^{(2)}]\smile\dots\smile[\xi^{(N)}]
\in H_0(\G)
\]
in terms of $\xi_i^{(j)}$. 
For $\sigma\in\mathcal{S}_N$ and $i=1,2,\dots,N{-}1$, 
we let 
\[
e(\sigma,i):=\sum_{k=i+1}^Ne_{\sigma(k)}\in\Z^N, 
\]
and let $e(\sigma,N):=0$. 
Then, for any $(g_1,g_2,\dots,g_N)\in E_\sigma$, one has 
\[
(\xi^{(1)}\smile\xi^{(2)}\smile\dots\smile\xi^{(N)})
(g_1,g_2,\dots,g_N)
=\prod_{i=1}^N
\left(\xi_{\sigma(i)}^{(i)}\circ\phi_{e(\sigma,i)}\right)(x), 
\]
where $x:=s(g_N)\in X$. 
As a result, we have the following theorem. 

\begin{theorem}
Let $\phi:\Z^N\curvearrowright X$ be an action 
on a Cantor set $X$ 
and let $\G:=\Z^N\times X$ be the transformation groupoid. 
Let $f\in C_c(\G^{(N)},\Z)$ be as above. 
When $\xi^{(1)},\xi^{(2)},\dots,\xi^{(N)}\in C(\G,\Z)$ are cocycles, 
\[
(f\frown\xi^{(1)}\smile\xi^{(2)}\smile\dots\smile\xi^{(N)})(x)
=\sum_{\sigma\in\mathcal{S}_N}\sgn(\sigma)\prod_{i=1}^N
\left(\xi_{\sigma(i)}^{(i)}\circ\phi_{e(\sigma,i)}\right)(x)
\]
holds for every $x\in X$. 
For example, when $N=2$, we get 
\[
f\frown\xi^{(1)}\smile\xi^{(2)}
=(\xi_{1}^{(1)}\circ\phi_{e_2})\cdot\xi_{2}^{(2)}
-(\xi_{2}^{(1)}\circ\phi_{e_1})\cdot\xi_{1}^{(2)}. 
\]
\end{theorem}

\subsection{The Penrose tiling}

Groupoids associated with tilings have been constructed 
in \cite{MR1359991} (see also \cite{KP00CRM}). 
For aperiodic, repetitive tilings with finite local complexity, 
the corresponding tiling groupoids are \'etale, minimal, 
and have unit spaces homeomorphic to the Cantor set. 
It is known that the cohomology of the groupoids 
are isomorphic to the \v{C}ech cohomology of the tiling space. 

In this subsection, 
we compute a cup product of cohomology of the Penrose tiling. 
The argument relies on the results in \cite{AP98ETDS}. 
Let $\Omega$ be the space of the Penrose tiling and 
let $\G$ be the associated ample groupoid. 
Our aim is to determine 
the cup product $H^1(\G)\times H^1(\G)\to H^2(\G)$. 
There exists a two-dimensional CW-complex $\Gamma$ 
and a continuous surjection $\gamma:\Gamma\to\Gamma$ 
such that $\Omega$ is homeomorphic to 
the inverse limit of $(\Gamma,\gamma)$. 
As mentioned in \cite[Section 10.4]{AP98ETDS}, 
$\gamma$ induces isomorphisms on $H^i(\Gamma)$, and so 
$H^i(\Omega)\cong H^i(\G)$ is naturally isomorphic to $H^i(\Gamma)$. 
Thus, it suffices to compute the cup product 
$H^1(\Gamma)\times H^1(\Gamma)\to H^2(\Gamma)$. 

\begin{figure}
\centering
\includegraphics[pagebox=cropbox,clip,scale=0.8]{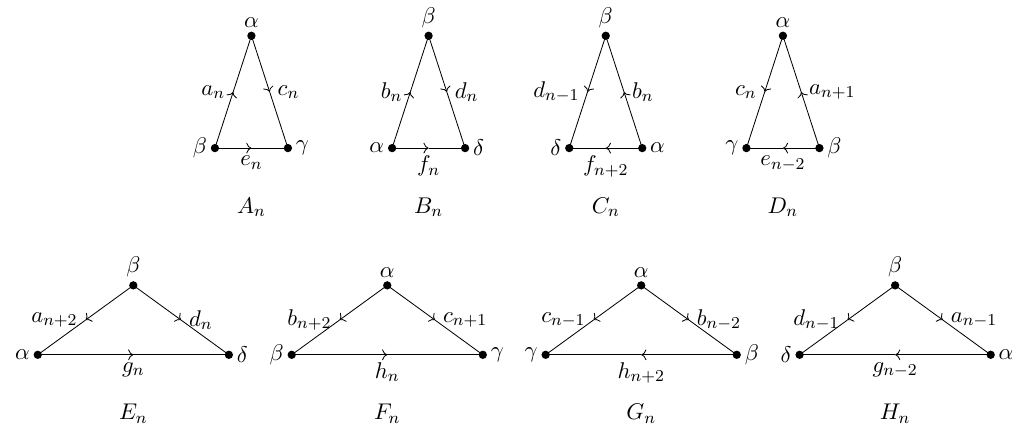}
\caption{the cell complex $\Gamma$}
\end{figure}

The cell complex $\Gamma$ has 40 faces, 40 edges, and four vertices: 
\[
F:=\{A_n,B_n,C_n,D_n,E_n,F_n,G_n,H_n\mid n=0,1,2,3,4\}, 
\]
\[
E:=\{a_n,b_n,c_n,d_n,e_n,f_n,g_n,h_n\mid n=0,1,2,3,4\}
\]
and $V:=\{\alpha,\beta,\gamma,\delta\}$. 
Figure 1 shows its faces and how they are glued together, 
where the index $n$ is understood modulo $5$. 
Let 
\[
\delta_1':\Z^V\to\Z^E\quad\text{and}\quad 
\delta_2':\Z^E\to\Z^F
\]
be the coboundary operators. 
Then, as computed in \cite[Section 10.4]{AP98ETDS}, we get 
\[
H^1(\Gamma)=\Ker\delta_2'/\Im\delta_1'\cong\Z^5
\]
and 
\[
H^2(\Gamma)=\Z^F/\Im\delta_2'\cong\Z^8. 
\]
We regard $\Z^E$ and $\Z^F$ 
as free abelian groups generated by $E$ and $F$, respectively. 
For $n=0,1,2,3,4$, we define $\xi_n\in\Z^E$ by 
\[
\xi_n:=a_n-b_{n+2}+c_{n+2}-d_{n+4}
+e_n+e_{n+2}-f_{n+2}-f_{n+4}-g_{n+3}-g_{n+4}+h_n+h_{n+1}. 
\]
Define $\eta\in\Z^E$ by 
\[
\eta:=\sum_{n=0}^4(a_n+e_n-g_n). 
\]
Then $[\xi_0],[\xi_1],[\xi_2],[\xi_3],[\eta]$ form a basis 
for $H^1(\Gamma)\cong\Z^5$, 
where the bracket means the equivalence class 
in $\Ker\delta_2'/\Im\delta_1'$. 
By the definition of the cup product for simplicial cohomology 
(see \cite{MR1867354} for instance), one has 
\[
\xi_0\smile\xi_1
=a_0\smile(-g_4)+(-b_2)\smile h_1
=H_1+G_4
\]
and 
\[
\xi_0\smile\xi_2
=a_0\smile c_4+(-b_2)\smile(-d_1)
=D_4+C_2. 
\]
Other products can be computed in the same way. 
Since $\Im\delta_2'\subset\Z^F$ is generated by 
\[
A_n-D_{n+2},\quad B_n-C_{n+3},\quad 
E_n-H_{n+2},\quad F_n-G_{n+3}
\]
\[
A_n-D_{n+4}-E_{n+3}+H_{n+1},\quad 
A_n-D_n+F_{n+4}-G_{n+1}
\]
\[
B_n-C_n-F_{n+3}+G_{n+2},\quad 
B_n-C_{n+1}+E_n-H_{n+1}, 
\]
we may choose 
\[
\left\{[A_0],[A_1],\dots,[A_4],[B_0],[E_0],[F_0]\right\}
\]
as a basis of $H^2(\Gamma)\cong\Z^8$. 
Therefore we get 
\[
[\xi_0]\smile[\xi_1]=[E_4]+[F_1]
=[A_0]+[A_1]-[A_2]-[A_3]+[E_0]+[F_0]
\]
and 
\[
[\xi_0]\smile[\xi_2]
=[A_2]+[B_4]=[A_3]+[B_0]. 
\]
Other products can be computed in the same way. 
In conclusion, we obtain the following theorem. 

\begin{theorem}
In the setting above, 
the cup product $H^1(\Gamma)\times H^1(\Gamma)\to H^2(\Gamma)$ 
is described as follows: 
\begin{align*}
[\xi_n]\smile[\xi_{n+1}]&=[A_0]+[A_1]-[A_2]-[A_3]+[E_0]+[F_0], \\
[\xi_n]\smile[\xi_{n+2}]&=[A_3]+[B_0], \\
[\xi_n]\smile[\xi_{n+3}]&=-[A_3]-[B_0], \\
[\xi_n]\smile[\eta]&=-[A_n]+[A_{n+2}]. 
\end{align*}
\end{theorem}

\subsection{The Ammann tiling}

In the same way as the previous subsection, 
we compute the cup product $H^1\times H^1\to H^2$ 
for the Ammann tiling. 
Let $\Omega$ be the space of the Ammann tiling and 
let $\G$ be the associated ample groupoid. 
There exists a two-dimensional CW-complex $\Gamma$ 
and a continuous surjection $\gamma:\Gamma\to\Gamma$ 
such that $\Omega$ is homeomorphic to 
the inverse limit of $(\Gamma,\gamma)$. 
As mentioned in \cite[Section 10.3]{AP98ETDS}, 
$\gamma$ induces isomorphisms on $H^i(\Gamma)$, and so 
$H^i(\Omega)\cong H^i(\G)$ is naturally isomorphic to $H^i(\Gamma)$. 
Thus, it suffices to compute the cup product 
$H^1(\Gamma)\times H^1(\Gamma)\to H^2(\Gamma)$. 

The cell complex $\Gamma$ has eight faces, eight edges, 
and three vertices: 
\[
F:=\{A,B,C,D,E,F,G,H\},\quad 
E:=\{a,b,c,d,e,f,g,h\}
\]
and $V:=\{\alpha,\beta,\gamma\}$. 
The left-hand side of \cite[Figure 4]{AP98ETDS} shows 
its faces and how they are glued together. 
Let 
\[
\delta_1':\Z^V\to\Z^E\quad\text{and}\quad 
\delta_2':\Z^E\to\Z^F
\]
be the coboundary operators. 
They are represented by the following matrices: 
\[
\delta_1'=
\begin{bmatrix}-1&0&1\\-1&0&1\\
-1&1&0\\-1&1&0\\
0&1&-1\\0&1&-1\\
1&0&-1\\1&0&-1\end{bmatrix},\qquad 
\delta_2'=
\begin{bmatrix}
1&1&-1&-1&1&1&0&0\\-1&-1&1&1&-1&-1&0&0\\
-1&-1&1&1&-1&-1&0&0\\1&1&-1&-1&1&1&0&0\\
0&0&1&1&-1&-1&1&1\\0&0&-1&-1&1&1&-1&-1\\
0&0&-1&-1&1&1&-1&-1\\0&0&1&1&-1&-1&1&1\\
\end{bmatrix}. 
\]
As computed in \cite[Section 10.3]{AP98ETDS}, we get 
\[
H^1(\Gamma)=\Ker\delta_2'/\Im\delta_1'\cong\Z^4
\]
and 
\[
H^2(\Gamma)=\Z^F/\Im\delta_2'\cong\Z^6. 
\]
We regard $\Z^E$ and $\Z^F$ 
as free abelian groups generated by $E$ and $F$, respectively. 
Define $\xi_1,\xi_2,\xi_3,\xi_4\in\Z^E$ by 
\[
\xi_1:=a-b,\quad \xi_2:=g-h, 
\]
\[
\xi_3:=a+c-g,\quad \xi_4:=a-e-g. 
\]
Then $[\xi_1],[\xi_2],[\xi_3],[\xi_4]$ form a basis 
for $H^1(\Gamma)\cong\Z^4$, 
where the bracket means the equivalence class 
in $\Ker\delta_2'/\Im\delta_1'$. 
Since $\Im\delta_2'\subset\Z^F$ is generated by 
\[
A-B-C+D\quad\text{ and }\quad E-F-G+H, 
\]
we may choose 
\[
\left\{[A],[B],[D],[E],[F],[H]\right\}
\]
as a basis of $H^2(\Gamma)\cong\Z^6$. 
In a similar fashion to the Penrose tiling, 
one can compute the cup product and obtain the following. 

\begin{theorem}
In the setting above, 
the cup product $H^1(\Gamma)\times H^1(\Gamma)\to H^2(\Gamma)$ 
is described as follows: 
\begin{align*}
[\xi_1]\smile[\xi_2]&=2([A]+[D]+[E]+[H]), \\
[\xi_1]\smile[\xi_3]&=-3([A]+[D])-2([E]+[H]), \\
[\xi_1]\smile[\xi_4]&=-([A]+[D]+[E]+[H]), \\
[\xi_2]\smile[\xi_3]&=-([A]+[D]+[E]+[H]), \\
[\xi_2]\smile[\xi_4]&=-2([A]+[D])-([E]+[H]), \\
[\xi_3]\smile[\xi_4]&=2([A]+[D])+([E]+[H]). 
\end{align*}
\end{theorem}

\newcommand{\noopsort}[1]{}

\end{document}